\documentclass{amsart}
\usepackage{amssymb}
\usepackage{amsmath}
\usepackage{amsfonts}

\newtheorem{theorem}{Theorem}
\theoremstyle{plain}

\newtheorem{corollary}{Corollary}

\newtheorem{definition}{Definition}

\newtheorem{remark}{Remark}

\numberwithin{equation}{section}
\input{tcilatex}

\begin{document}
\title[On Hadamard Type Inequalities]{On Hadamard Type Inequalities
Involving Several Kind of Convexity }
\author{Erhan SET$^{\star \clubsuit }$}
\address{$^{\clubsuit }$Atat\"{u}rk University, K.K. Education Faculty,
Department of Mathematics, 25240, Campus, Erzurum, Turkey}
\email{erhanset@yahoo.com}
\thanks{$^{\star }$corresponding author}
\author{M. Emin \"{O}zdemir$^{\blacksquare }$}
\address{$^{\blacksquare }$Graduate School of Natural and Applied Sciences, A%
\u{g}r\i\ \.{I}brahim \c{C}e\c{c}en University, A\u{g}r\i , Turkey}
\email{emos@atauni.edu.tr}
\author{Sever S. Dragomir$^{\blacktriangle }$}
\address{$^{\blacktriangle }$Research Group in Mathematical Inequalities \&
Applications\\
School of Engineering \& Science\\
Victoria University, PO Box 14428\\
Melbourne City, MC 8001, Australia.}
\email{sever.dragomir@vu.edu.au}
\urladdr{http://www.staff.vu.edu.au/rgmia/dragomir/}

\begin{abstract}
In this paper, we not only give the extensions of the results given in \cite%
{GP} by Gill et al. for log-convex functions, but also obtain some new
Hadamard type inequalities for log-convex, $m-$convex and $\left( \alpha
,m\right) $-convex functions.
\end{abstract}

\maketitle

\section{Introduction}

The following inequality is well known in the literature as Hadamard's
inequality:%
\begin{equation}
f\left( \frac{a+b}{2}\right) \leq \frac{1}{b-a}\int_{a}^{b}f\left( x\right)
dx\leq \frac{f\left( a\right) +f\left( b\right) }{2},  \label{e.0}
\end{equation}%
where $f:I\rightarrow R$ is a convex function on the interval $I$ of real
numbers and $a,b\in I$ with $a<b.$ This inequality is one of the most useful
inequalities in mathematical analysis. For new proofs, note worthy
extension, generalizations and numerous applications on this inequality, see
(\cite{AD}, \cite{MCXHZ}, \cite{CD}, \cite{DP}, \cite{MPF}, \cite{SOD})
where further references are given.

Let $I$ be on interval in $R$. Then $f:I\rightarrow R$ is said to be convex
if for all $x,y\in I$ and $\lambda \in \left[ 0,1\right] $,%
\begin{equation*}
f\left( \lambda x+\left( 1-\lambda \right) y\right) \leq \lambda f\left(
x\right) +\left( 1-\lambda \right) f\left( y\right)
\end{equation*}%
(see \cite[P.1]{MPF}). Geometrically, this means that if $K,L$ and $M$ are
three distinct points on the graph of $f$ with $L$ between $K$ and $M$, then 
$L$ is on or below chord $KM.$

Recall that a function $f:I\rightarrow \left( 0,\infty \right) $ is said to
be log-convex function, if for all $x,y\in I$ and $t\in \left[ 0,1\right] $,
one has the inequality (see \cite[P.3]{MPF})%
\begin{equation}
f\left( tx+\left( 1-t\right) y\right) \leq \left[ f\left( x\right) \right]
^{t}\left[ f\left( y\right) \right] ^{\left( 1-t\right) }.  \label{e.1}
\end{equation}%
It is said to be log-concave if the inequality in (\ref{e.1}) is reversed.

In \cite{GT}, G. Toader defined $m$-convexity as follows:

\begin{definition}
\label{d1} The function $f:\left[ 0,b\right] \rightarrow R$, $b>0$ is said
to be $m$-convex, where $m\in \left[ 0,1\right] $, if we have 
\begin{equation*}
f\left( tx+m\left( 1-t\right) y\right) \leq tf\left( x\right) +m\left(
1-t\right) f\left( y\right)
\end{equation*}%
for all $x,y\in \left[ 0,b\right] $ and $t\in \left[ 0,1\right] .$We say
that $f$ is $m-$concave if $-f$ is $m-$convex.
\end{definition}

Denote by $K_{m}\left( b\right) $ the class of all $m-$convex functions on $%
\left[ 0,b\right] $ for which $f\left( 0\right) \leq 0.$ Obviously, if we
choose $m=1,$ definition (\ref{d1}) recaptures the concept of standard
convex functions on $\left[ 0,b\right] .$

In \cite{VGM} , V. G. Mihe\c{s}an defined $\left( \alpha ,m\right) -$
convexity as in the following:

\begin{definition}
\label{d.1.2} The function $f:\left[ 0,b\right] \rightarrow 
\mathbb{R}
$ , $b>0$, is said to be $\left( \alpha ,m\right) -$ convex, where $\left(
\alpha ,m\right) \in \left[ 0,1\right] ^{2}$, if we have%
\begin{equation*}
f(tx+m(1-t)y)\leq t^{\alpha }f(x)+m(1-t^{\alpha })f(y)
\end{equation*}%
for all $x,y\in \left[ 0,b\right] $ and $t\in \left[ 0,1\right] $.
\end{definition}

Denote by $K_{m}^{\alpha }(b)$ the class of all $\left( \alpha ,m\right) -$%
convex functions on $\left[ 0,b\right] $ for which $f(0)\leq 0$. It can be
easily seen that for $\left( \alpha ,m\right) =\left( 1,m\right) ,$ $\left(
\alpha ,m\right) -$ convexity reduces to $m-$ convexity and for $\left(
\alpha ,m\right) =\left( 1,1\right) $, $\left( \alpha ,m\right) -$ convexity
reduces to the concept of usual convexity defined on $\left[ 0,b\right] $ , $%
b>0$.

For recent results and generalizations concerning $m-$convex and $\left(
\alpha ,m\right) -$ convex functions, see (\cite{BOP}, \cite{BPR}, \cite{OAS}%
).

In \cite{GP}, P.M. Gill et al. established the following results:

\begin{theorem}
\label{A} Let $f$ be a positive, $\log $-convex function on $\left[ a,b%
\right] $. Then%
\begin{equation}
\frac{1}{b-a}\int_{a}^{b}f\left( t\right) dt\leq L\left( f\left( a\right)
,f\left( b\right) \right)  \label{e.2}
\end{equation}%
where 
\begin{equation*}
L\left( p,q\right) =\frac{p-q}{\ln p-\ln q}\;~\;\;(p\neq q)
\end{equation*}%
is the Logarithmic mean of the positive real numbers $p,q$ \ (for $p=q$, we
put $L\left( p,p\right) =p$).

For $f$ a positive $\log $-concave function, the inequality is reversed.
\end{theorem}

\begin{corollary}
\label{c1} Let $f$ be positive $\log $-convex functions on $\left[ a,b\right]
$. Then%
\begin{multline*}
\dfrac{1}{b-a}\int_{a}^{b}f\left( t\right) dt \\
\leq \underset{x\in \left[ a,b\right] }{\min }\dfrac{\left( x-a\right)
L\left( f\left( a\right) ,f\left( x\right) \right) +\left( b-x\right)
L\left( f\left( x\right) ,f\left( b\right) \right) }{b-a}.
\end{multline*}%
If $f$ is a positive $\log $-concave function, then%
\begin{multline*}
\dfrac{1}{b-a}\int_{a}^{b}f\left( x\right) dx \\
\geq \underset{x\in \left[ a,b\right] }{\max }\dfrac{\left( x-a\right)
L\left( f\left( a\right) ,f\left( x\right) \right) +\left( b-x\right)
L\left( f\left( x\right) ,f\left( b\right) \right) }{b-a}.
\end{multline*}
\end{corollary}

For some recent results related to the Hadamard's inequalities involving two 
$\log $-convex functions, see \cite{BGP} and the references cited therein.
The main purpose of this paper is to establish the general version of the
inequalities (\ref{e.2}) and new Hadamard type inequalities involving two $%
\log $-convex functions or two $m$-convex functions or two $\left( \alpha
,m\right) $-convex functions using elementary analysis.

\section{Main Results}

We start with the following Theorem.

\begin{theorem}
\label{t1} Let $f_{i}:I\subset R\rightarrow \left( 0,\infty \right) $ $%
\left( i=1,2,...,n\right) $ be $\log $-convex functions on $I$ and $a,b\in I$
with $a<b$. Then the following inequality holds:%
\begin{equation}
\frac{1}{b-a}\int_{a}^{b}\prod\limits_{i=1}^{n}f_{i}\left( x\right) dx\leq
L\left( \prod\limits_{i=1}^{n}f_{i}\left( a\right)
,\prod\limits_{i=1}^{n}f_{i}\left( b\right) \right)  \label{E3}
\end{equation}%
where $L$ is a logarithmic mean of positive real numbers.

For $f$ a positive $\log $-concave function, the inequality is reversed.
\end{theorem}

\begin{proof}
Since $f_{i}$ $\left( i=1,2,...,n\right) $ are $\log $-convex functions on $%
I $, we have 
\begin{equation}
f_{i}\left( ta+\left( 1-t\right) b\right) \leq \left[ f_{i}\left( a\right) %
\right] ^{t}\left[ f_{i}\left( b\right) \right] ^{\left( 1-t\right) }
\label{E4}
\end{equation}%
for all $a,b\in I$ and $t\in \left[ 0,1\right] .$ Writing (\ref{E4}) for $%
i=1,2,...,n$, multiplying the resulting inequalities it is easy to observe
that%
\begin{eqnarray}
\prod\limits_{i=1}^{n}f_{i}\left( ta+\left( 1-t\right) b\right) &\leq & 
\left[ \prod\limits_{i=1}^{n}f_{i}\left( a\right) \right] ^{t}\left[
\prod\limits_{i=1}^{n}f_{i}\left( b\right) \right] ^{\left( 1-t\right) }
\label{E6} \\
&=&\prod\limits_{i=1}^{n}f_{i}\left( b\right) \left[ \prod\limits_{i=1}^{n}%
\dfrac{f_{i}\left( a\right) }{f_{i}\left( b\right) }\right] ^{t}  \notag
\end{eqnarray}%
for all $a,b\in I$ and $t\in \left[ 0,1\right] .$

Integrating inequality (\ref{E6}) on $\left[ 0,1\right] $ over $t$, we get%
\begin{equation*}
\int_{0}^{1}\prod\limits_{i=1}^{n}f_{i}\left( ta+\left( 1-t\right) b\right)
dt\leq \prod\limits_{i=1}^{n}f_{i}\left( b\right) \int_{0}^{1} \left[
\prod\limits_{i=1}^{n}\dfrac{f_{i}\left( a\right) }{f_{i}\left( b\right) }%
\right] ^{t}dt.
\end{equation*}%
As%
\begin{equation*}
\int_{0}^{1}\prod\limits_{i=1}^{n}f_{i}\left( ta+\left( 1-t\right) b\right)
dt=\frac{1}{b-a}\int_{a}^{b}\prod\limits_{i=1}^{n}f_{i}\left( x\right) dx
\end{equation*}%
and%
\begin{equation*}
\int_{0}^{1}\left[ \prod\limits_{i=1}^{n}\dfrac{f_{i}\left( a\right) }{%
f_{i}\left( b\right) }\right] ^{t}dt=\frac{1}{\prod\limits_{i=1}^{n}f_{i}%
\left( b\right) }L\left( \prod\limits_{i=1}^{n}f_{i}\left( a\right)
,\prod\limits_{i=1}^{n}f_{i}\left( b\right) \right) ,
\end{equation*}%
the theorem is proved.
\end{proof}

\begin{remark}
\label{r1} By taking $i=1$ and $f_{1}=f$ in Theorem \ref{t1}$,$ we obtain (%
\ref{e.2}).
\end{remark}

\begin{corollary}
\label{c2} Let $f_{i}:I\subset R\rightarrow \left( 0,\infty \right) $ $%
(i=1,2,...,n)$ be $\log $-convex functions on $I$ and $a,b\in I$ with $a<b$.
Then%
\begin{eqnarray}
&&\dfrac{1}{b-a}\int_{a}^{b}\prod\limits_{i=1}^{n}f_{i}\left( x\right) dx
\label{E7} \\
&\leq &\underset{x\in \left[ a,b\right] }{\min }\dfrac{\left( x-a\right)
L\left( \prod\limits_{i=1}^{n}f_{i}\left( a\right)
,\prod\limits_{i=1}^{n}f_{i}\left( x\right) \right) +\left( b-x\right)
L\left( \prod\limits_{i=1}^{n}f_{i}\left( x\right)
,\prod\limits_{i=1}^{n}f_{i}\left( b\right) \right) }{b-a}.  \notag
\end{eqnarray}%
If $f_{i}$ $(i=1,2,...,n)$ are a positive $\log $-concave functions, then%
\begin{align}
& \dfrac{1}{b-a}\int_{a}^{b}\prod\limits_{i=1}^{n}f_{i}\left( x\right) dx
\label{E8} \\
& \geq \underset{x\in \left[ a,b\right] }{\max }\dfrac{\left( x-a\right)
L\left( \prod\limits_{i=1}^{n}f_{i}\left( a\right)
,\prod\limits_{i=1}^{n}f_{i}\left( x\right) \right) +\left( b-x\right)
L\left( \prod\limits_{i=1}^{n}f_{i}\left( x\right)
,\prod\limits_{i=1}^{n}f_{i}\left( b\right) \right) }{b-a}.  \notag
\end{align}
\end{corollary}

\begin{proof}
Let $f_{i}$ $(i=1,2,...,n)$ be a positive $\log $-convex functions. Then by
Theorem \ref{t1} we have that%
\begin{eqnarray*}
&&\int_{a}^{b}\prod\limits_{i=1}^{n}f_{i}\left( t\right) dt \\
&=&\int_{a}^{x}\prod\limits_{i=1}^{n}f_{i}\left( t\right)
dt+\int_{x}^{b}\prod\limits_{i=1}^{n}f_{i}\left( t\right) dt \\
&\leq &\left( x-a\right) L\left( \prod\limits_{i=1}^{n}f_{i}\left( a\right)
,\prod\limits_{i=1}^{n}f_{i}\left( x\right) \right) +\left( b-x\right)
L\left( \prod\limits_{i=1}^{n}f_{i}\left( x\right)
,\prod\limits_{i=1}^{n}f_{i}\left( b\right) \right)
\end{eqnarray*}%
for all $x\in \left[ a,b\right] $, whence (\ref{E7}). Similarly we can prove
(\ref{E8}).
\end{proof}

\begin{remark}
\label{r2} By taking $i=1$ and $f_{1}=f$ in (\ref{E7})and (\ref{E8})$,$ we
obtain the inequalities of Corollary \ref{c1}.
\end{remark}

We will now point out some new results of the Hadamard type for log-convex, $%
m-$convex and $\left( \alpha ,m\right) $-convex functions, respectively.

\begin{theorem}
\label{t2} Let $f,g:I\rightarrow \left( 0,\infty \right) $ be $\log $-convex
functions on $I$ and $a,b\in I$ with $a<b.$ Then the following inequalities
hold:%
\begin{eqnarray}
&&f\left( \dfrac{a+b}{2}\right) g\left( \dfrac{a+b}{2}\right)  \label{E9} \\
&\leq &\dfrac{1}{2}\left\{ \dfrac{1}{b-a}\int_{a}^{b}\left[ f\left( x\right)
f\left( a+b-x\right) +g\left( x\right) g\left( a+b-x\right) \right]
dx\right\}  \notag \\
&\leq &\dfrac{f\left( a\right) f\left( b\right) +g\left( a\right) g\left(
b\right) }{2}.  \notag
\end{eqnarray}
\end{theorem}

\begin{proof}
We can write%
\begin{equation}
\frac{a+b}{2}=\frac{ta+\left( 1-t\right) b}{2}+\frac{\left( 1-t\right) a+tb}{%
2}.  \label{E10}
\end{equation}%
Using the elementary inequality $cd\leq \frac{1}{2}\left[ c^{2}+d^{2}\right]
\;\;$($c,d\geq 0$ reals) and equality (\ref{E10}), we have%
\begin{eqnarray}
&&f\left( \dfrac{a+b}{2}\right) g\left( \dfrac{a+b}{2}\right)  \label{E11} \\
&\leq &\frac{1}{2}\left[ f^{2}\left( \dfrac{a+b}{2}\right) +g^{2}\left( 
\dfrac{a+b}{2}\right) \right]  \notag \\
&=&\frac{1}{2}\left[ f^{2}\left( \dfrac{ta+\left( 1-t\right) b}{2}+\dfrac{%
\left( 1-t\right) a+tb}{2}\right) \right.  \notag \\
&&\left. +g^{2}\left( \dfrac{ta+\left( 1-t\right) b}{2}+\dfrac{\left(
1-t\right) a+tb}{2}\right) \right]  \notag \\
&\leq &\dfrac{1}{2}\left\{ \left[ \left( f\left( ta+\left( 1-t\right)
b\right) \right) ^{\frac{1}{2}}\right] ^{2}\left[ \left( f\left( \left(
1-t\right) a+tb\right) \right) ^{\frac{1}{2}}\right] ^{2}\right.  \notag \\
&&\left. +\left[ \left( g\left( ta+\left( 1-t\right) b\right) \right) ^{%
\frac{1}{2}}\right] ^{2}\left[ \left( g\left( \left( 1-t\right) a+tb\right)
\right) ^{\frac{1}{2}}\right] ^{2}\right\}  \notag \\
&=&\dfrac{1}{2}\left[ f\left( ta+\left( 1-t\right) b\right) f\left( \left(
1-t\right) a+tb\right) \right.  \notag \\
&&\left. +g\left( ta+\left( 1-t\right) b\right) g\left( \left( 1-t\right)
a+tb\right) \right] .  \notag
\end{eqnarray}%
Since $f,g$ are $\log $-convex functions, we obtain%
\begin{eqnarray}
&&\frac{1}{2}\left[ f\left( ta+\left( 1-t\right) b\right) f\left( \left(
1-t\right) a+tb\right) \right.  \label{E12} \\
&&\left. +g\left( ta+\left( 1-t\right) b\right) g\left( \left( 1-t\right)
a+tb\right) \right]  \notag \\
&\leq &\left\{ \dfrac{1}{2}\left[ f\left( a\right) \right] ^{t}\left[
f\left( b\right) \right] ^{\left( 1-t\right) }\left[ f\left( a\right) \right]
^{\left( 1-t\right) }\left[ f\left( b\right) \right] ^{t}\right.  \notag \\
&&\left. +\left[ g\left( a\right) \right] ^{t}\left[ g\left( b\right) \right]
^{\left( 1-t\right) }\left[ g\left( a\right) \right] ^{\left( 1-t\right) }%
\left[ g\left( b\right) \right] ^{t}\right\}  \notag \\
&=&\dfrac{f\left( a\right) f\left( b\right) +g\left( a\right) g\left(
b\right) }{2}  \notag
\end{eqnarray}%
for all $a,b\in I$ and $t\in \left[ 0,1\right] $.

Rewriting (\ref{E11}) and (\ref{E12}), we have%
\begin{eqnarray}
f\left( \dfrac{a+b}{2}\right) g\left( \dfrac{a+b}{2}\right) &\leq &\dfrac{1}{%
2}\left[ f\left( ta+\left( 1-t\right) b\right) f\left( \left( 1-t\right)
a+tb\right) \right.  \label{E13} \\
&&\left. +g\left( ta+\left( 1-t\right) b\right) g\left( \left( 1-t\right)
a+tb\right) \right]  \notag
\end{eqnarray}%
and%
\begin{eqnarray}
&&\dfrac{1}{2}\left[ f\left( ta+\left( 1-t\right) b\right) f\left( \left(
1-t\right) a+tb\right) +g\left( ta+\left( 1-t\right) b\right) g\left( \left(
1-t\right) a+tb\right) \right]  \label{E14} \\
&\leq &\dfrac{f\left( a\right) f\left( b\right) +g\left( a\right) g\left(
b\right) }{2}.  \notag
\end{eqnarray}%
Integrating both sides of (\ref{E13}) and (\ref{E14}) on $\left[ 0,1\right] $
over $t$, respectively, we obtain%
\begin{eqnarray}
&&f\left( \dfrac{a+b}{2}\right) g\left( \dfrac{a+b}{2}\right)  \label{E15} \\
&\leq &\dfrac{1}{2}\left[ \dfrac{1}{b-a}\int_{a}^{b}\left[ f\left( x\right)
f\left( a+b-x\right) +g\left( x\right) g\left( a+b-x\right) \right] dx\right]
\notag
\end{eqnarray}%
and%
\begin{eqnarray}
&&\dfrac{1}{2}\left[ \dfrac{1}{b-a}\int_{a}^{b}\left[ f\left( x\right)
f\left( a+b-x\right) +g\left( x\right) g\left( a+b-x\right) \right] dx\right]
\label{E16} \\
&\leq &\dfrac{f\left( a\right) f\left( b\right) +g\left( a\right) g\left(
b\right) }{2}.  \notag
\end{eqnarray}%
Combining (\ref{E15}) and (\ref{E16}), we get the desired inequalities (\ref%
{E9}). The proof is complete.
\end{proof}

\begin{theorem}
\label{t3}Let $f,g:I\rightarrow \left( 0,\infty \right) $ be $\log $-convex
functions on $I$ and $a,b\in I$ with $a<b.$ Then the following inequalities
hold:%
\begin{eqnarray}
&&2f\left( \dfrac{a+b}{2}\right) g\left( \dfrac{a+b}{2}\right)  \label{E17}
\\
&\leq &\dfrac{1}{b-a}\int_{a}^{b}\left[ f^{2}\left( x\right) +g^{2}\left(
x\right) \right] dx  \notag \\
&\leq &\dfrac{f\left( a\right) +f\left( b\right) }{2}L\left( f\left(
a\right) ,f\left( b\right) \right) +\dfrac{g\left( a\right) +g\left(
b\right) }{2}L\left( g\left( a\right) ,g\left( b\right) \right)  \notag
\end{eqnarray}%
where $L\left( .\;,\;.\right) $ is a logarithmic mean of positive real
numbers.
\end{theorem}

\begin{proof}
From the inequality (\ref{E13}), we have%
\begin{eqnarray*}
f\left( \dfrac{a+b}{2}\right) g\left( \dfrac{a+b}{2}\right) &\leq &\dfrac{1}{%
2}\left[ f\left( ta+\left( 1-t\right) b\right) f\left( \left( 1-t\right)
a+tb\right) \right. \\
&&+\left. g\left( ta+\left( 1-t\right) b\right) g\left( \left( 1-t\right)
a+tb\right) \right]
\end{eqnarray*}%
for all $a,b\in I$ and $t\in \left[ 0,1\right] .$

Using the elementary inequality $cd\leq \frac{1}{2}\left[ c^{2}+d^{2}\right]
\;\;$($c,d\geq 0$ reals) on the right side of the above inequality, we have%
\begin{eqnarray}
f\left( \dfrac{a+b}{2}\right) g\left( \dfrac{a+b}{2}\right) &\leq &\dfrac{1}{%
4}\left[ f^{2}\left( ta+\left( 1-t\right) b\right) +f^{2}\left( \left(
1-t\right) a+tb\right) \right.  \label{E18} \\
&&+\left. g^{2}\left( ta+\left( 1-t\right) b\right) +g^{2}\left( \left(
1-t\right) a+tb\right) \right] .  \notag
\end{eqnarray}%
Since $f,g$ are $\log $-convex functions, then we get%
\begin{eqnarray}
&&\left[ f^{2}\left( ta+\left( 1-t\right) b\right) +f^{2}\left( \left(
1-t\right) a+tb\right) \right.  \label{E19} \\
&&\left. +g^{2}\left( ta+\left( 1-t\right) b\right) +g^{2}\left( \left(
1-t\right) a+tb\right) \right]  \notag \\
&\leq &\left\{ \left[ f\left( a\right) \right] ^{2t}\left[ f\left( b\right) %
\right] ^{\left( 2-2t\right) }+\left[ f\left( a\right) \right] ^{\left(
2-2t\right) }\left[ f\left( b\right) \right] ^{2t}\right.  \notag \\
&&\left. +\left[ g\left( a\right) \right] ^{2t}\left[ g\left( b\right) %
\right] ^{\left( 2-2t\right) }+\left[ g\left( a\right) \right] ^{\left(
2-2t\right) }\left[ g\left( b\right) \right] ^{2t}\right\}  \notag \\
&=&\left[ f^{2}\left( b\right) \left[ \dfrac{f\left( a\right) }{f\left(
b\right) }\right] ^{2t}+f^{2}\left( a\right) \left[ \dfrac{f\left( b\right) 
}{f\left( a\right) }\right] ^{2t}\right.  \notag \\
&&\left. +g^{2}\left( b\right) \left[ \dfrac{g\left( a\right) }{g\left(
b\right) }\right] ^{2t}+g^{2}\left( a\right) \left[ \dfrac{g\left( b\right) 
}{g\left( a\right) }\right] ^{2t}\right] .  \notag
\end{eqnarray}%
Integrating both sides of (\ref{E18}) and (\ref{E19}) on $\left[ 0,1\right] $
over $t$, respectively, we obtain 
\begin{equation}
2f\left( \frac{a+b}{2}\right) g\left( \frac{a+b}{2}\right) \leq \frac{1}{b-a}%
\int_{a}^{b}\left[ f^{2}\left( x\right) +g^{2}\left( x\right) \right] dx
\label{E20}
\end{equation}%
and%
\begin{eqnarray}
&&\dfrac{1}{b-a}\int_{a}^{b}\left[ f^{2}\left( x\right) +g^{2}\left(
x\right) \right] dx  \label{ek1} \\
&\leq &\dfrac{1}{2}\left( f^{2}\left( b\right) \int_{0}^{1}\left[ \dfrac{%
f\left( a\right) }{f\left( b\right) }\right] ^{2t}dt+f^{2}\left( a\right)
\int_{0}^{1}\left[ \dfrac{f\left( b\right) }{f\left( a\right) }\right]
^{2t}dt\right.  \notag \\
&&\left. +g^{2}\left( b\right) \int_{0}^{1}\left[ \dfrac{g\left( a\right) }{%
g\left( b\right) }\right] ^{2t}dt+g^{2}\left( a\right) \int_{0}^{1}\left[ 
\dfrac{g\left( b\right) }{g\left( a\right) }\right] ^{2t}dt\right)  \notag \\
&=&\dfrac{1}{2}\left( f^{2}\left( b\right) \left[ \dfrac{\left[ \frac{%
f\left( a\right) }{f\left( b\right) }\right] ^{2t}}{2\log \frac{f\left(
a\right) }{f\left( b\right) }}\right] _{0}^{1}+f^{2}\left( a\right) \left[ 
\dfrac{\left[ \frac{f\left( b\right) }{f\left( a\right) }\right] ^{2t}}{%
2\log \frac{f\left( b\right) }{f\left( a\right) }}\right] _{0}^{1}\right. 
\notag \\
&&\left. +g^{2}\left( b\right) \left[ \dfrac{\left[ \frac{g\left( a\right) }{%
g\left( b\right) }\right] ^{2t}}{2\log \frac{g\left( a\right) }{g\left(
b\right) }}\right] _{0}^{1}+g^{2}\left( a\right) \left[ \dfrac{\left[ \frac{%
g\left( b\right) }{g\left( a\right) }\right] ^{2t}}{2\log \frac{g\left(
b\right) }{g\left( a\right) }}\right] _{0}^{1}\right)  \notag
\end{eqnarray}%
\begin{eqnarray*}
&=&\dfrac{1}{2}\left( \dfrac{f^{2}\left( a\right) -f^{2}\left( b\right) }{%
2\left( \log f\left( a\right) -\log f\left( b\right) \right) }+\dfrac{%
f^{2}\left( b\right) -f^{2}\left( a\right) }{2\left( \log f\left( b\right)
-\log f\left( a\right) \right) }\right. \\
&&\left. +\dfrac{g^{2}\left( a\right) -g^{2}\left( b\right) }{2\left( \log
g\left( a\right) -\log g\left( b\right) \right) }+\dfrac{g^{2}\left(
b\right) -g^{2}\left( a\right) }{2\left( \log g\left( b\right) -\log g\left(
a\right) \right) }\right) \\
&=&\dfrac{1}{2}\left( \dfrac{f\left( a\right) +f\left( b\right) }{2}L\left(
f\left( a\right) ,f\left( b\right) \right) +\dfrac{f\left( a\right) +f\left(
b\right) }{2}L\left( f\left( b\right) ,f\left( a\right) \right) \right. \\
&&\left. +\dfrac{g\left( a\right) +g\left( b\right) }{2}L\left( g\left(
a\right) ,g\left( b\right) \right) +\dfrac{g\left( a\right) +g\left(
b\right) }{2}L\left( g\left( b\right) ,g\left( a\right) \right) \right) \\
&=&\left\{ \dfrac{f\left( a\right) +f\left( b\right) }{2}L\left( f\left(
a\right) ,f\left( b\right) \right) +\dfrac{g\left( a\right) +g\left(
b\right) }{2}L\left( g\left( a\right) ,g\left( b\right) \right) \right\} .
\end{eqnarray*}%
Combining (\ref{E20}) and (\ref{ek1}), we get the required inequalities (\ref%
{E17}). The proof is complete.
\end{proof}

\begin{theorem}
\label{t4}Let $f,g:\left[ 0,\infty \right) \rightarrow \left[ 0,\infty
\right) $ be such that $fg$ is in $L^{1}\left( \left[ a,b\right] \right) $,
where $0\leq a<b<\infty $. If $f$ is non-increasing $m_{1}-$convex function
and $g$ is non-increasing $m_{2}-$convex function on $\left[ a,b\right] $
for some fixed $m_{1},m_{2}\in \left( 0,1\right] ,$ then the following
inequality holds:%
\begin{equation}
\dfrac{1}{b-a}\int_{a}^{b}f\left( x\right) g\left( x\right) dx\leq \min
\left\{ S_{1},S_{2}\right\}  \label{E22}
\end{equation}%
where%
\begin{eqnarray*}
S_{1} &=&\dfrac{1}{6}\left[ \left( f^{2}\left( a\right) +g^{2}\left(
a\right) \right) +m_{1}f\left( a\right) f\left( \dfrac{b}{m_{1}}\right)
\right. \\
&&+\left. m_{2}g\left( a\right) g\left( \dfrac{b}{m_{2}}\right)
+m_{1}^{2}f^{2}\left( \dfrac{b}{m_{1}}\right) +m_{2}^{2}g^{2}\left( \dfrac{b%
}{m_{2}}\right) \right]
\end{eqnarray*}%
\begin{eqnarray*}
S_{2} &=&\dfrac{1}{6}\left[ \left( f^{2}\left( b\right) +g^{2}\left(
b\right) \right) +m_{1}f\left( b\right) f\left( \dfrac{a}{m_{1}}\right)
\right. \\
&&+\left. m_{2}g\left( b\right) g\left( \dfrac{a}{m_{2}}\right)
+m_{1}^{2}f^{2}\left( \dfrac{a}{m_{1}}\right) +m_{2}^{2}g^{2}\left( \dfrac{a%
}{m_{2}}\right) \right] .
\end{eqnarray*}
\end{theorem}

\begin{proof}
Since $f$ is $m_{1}$-convex function and $g$ is $m_{2}$-convex function, we
have 
\begin{equation}
f\left( ta+\left( 1-t\right) b\right) \leq tf\left( a\right) +m_{1}\left(
1-t\right) f\left( \frac{b}{m_{1}}\right)  \label{E23}
\end{equation}%
and%
\begin{equation}
g\left( ta+\left( 1-t\right) b\right) \leq tg\left( a\right) +m_{2}\left(
1-t\right) g\left( \frac{b}{m_{2}}\right)  \label{E24}
\end{equation}%
for all $t\in \left[ 0,1\right] $. It is easy to observe that%
\begin{equation}
\int_{a}^{b}f\left( x\right) g\left( x\right) dx=\left( b-a\right)
\int_{0}^{1}f\left( ta+\left( 1-t\right) b\right) g\left( ta+\left(
1-t\right) b\right) dt.  \label{E25}
\end{equation}%
Using the elementary inequality $cd\leq \frac{1}{2}\left( c^{2}+d^{2}\right)
\;$($c,d\geq 0$ reals), (\ref{E23}) and (\ref{E24}) on the right side of (%
\ref{E25}) and making the charge of variable and since $f,g$ is
non-increasing, we have%
\begin{eqnarray}
&&\int_{a}^{b}f\left( x\right) g\left( x\right) dx  \label{E26} \\
&\leq &\dfrac{1}{2}\left( b-a\right) \int_{0}^{1}\left[ \left\{ f\left(
ta+\left( 1-t\right) b\right) \right\} ^{2}+\left\{ g\left( ta+\left(
1-t\right) b\right) \right\} ^{2}\right] dt  \notag \\
&\leq &\dfrac{1}{2}\left( b-a\right) \int_{0}^{1}\left[ \left( tf\left(
a\right) +m_{1}\left( 1-t\right) f\left( \dfrac{b}{m_{1}}\right) \right)
^{2}\right.  \notag \\
&&\left. +\left( tg\left( a\right) +m_{2}\left( 1-t\right) g\left( \dfrac{b}{%
m_{2}}\right) \right) ^{2}\right] dt  \notag \\
&=&\dfrac{1}{2}\left( b-a\right) \left[ \dfrac{1}{3}f^{2}\left( a\right) +%
\dfrac{1}{3}m_{1}^{2}f^{2}\left( \dfrac{b}{m_{1}}\right) +\dfrac{1}{3}%
m_{1}f\left( a\right) f\left( \dfrac{b}{m_{1}}\right) \right.  \notag \\
&&\left. +\dfrac{1}{3}g^{2}\left( a\right) +\dfrac{1}{3}m_{2}^{2}g^{2}\left( 
\dfrac{b}{m_{2}}\right) +\dfrac{1}{3}m_{2}g\left( a\right) g\left( \dfrac{b}{%
m_{2}}\right) \right]  \notag \\
&=&\dfrac{\left( b-a\right) }{6}\left[ \left( f^{2}\left( a\right)
+g^{2}\left( a\right) \right) +m_{1}f\left( a\right) f\left( \dfrac{b}{m_{1}}%
\right) \right.  \notag \\
&&\left. +m_{2}g\left( a\right) g\left( \dfrac{b}{m_{2}}\right)
+m_{1}^{2}f^{2}\left( \dfrac{b}{m_{1}}\right) +m_{2}^{2}g^{2}\left( \dfrac{b%
}{m_{2}}\right) \right] .  \notag
\end{eqnarray}%
Analogously we obtain%
\begin{eqnarray}
&&\int_{a}^{b}f\left( x\right) g\left( x\right) dx  \label{E27} \\
&\leq &\dfrac{\left( b-a\right) }{6}\left[ \left( f^{2}\left( b\right)
+g^{2}\left( b\right) \right) +m_{1}f\left( b\right) f\left( \dfrac{a}{m_{1}}%
\right) \right.  \notag \\
&&\left. +m_{2}g\left( b\right) g\left( \dfrac{a}{m_{2}}\right)
+m_{1}^{2}f^{2}\left( \dfrac{a}{m_{1}}\right) +m_{2}^{2}g^{2}\left( \dfrac{a%
}{m_{2}}\right) \right] .  \notag
\end{eqnarray}%
Rewriting (\ref{E26}) and (\ref{E27}), we get the required inequality in (%
\ref{E22}). The proof is complete.
\end{proof}

\begin{theorem}
\label{t5} Let $f,g:\left[ 0,\infty \right) \rightarrow \left[ 0,\infty
\right) $ be such that $fg$ is in $L^{1}\left( \left[ a,b\right] \right) $,
where $0\leq a<b<\infty $. If $f$ is non-increasing $\left( \alpha
_{1},m_{1}\right) -$convex function and $g$ is non-increasing $\left( \alpha
_{2},m_{2}\right) -$convex function on $\left[ a,b\right] $ for some fixed $%
\alpha _{1},m_{1},\alpha _{2},m_{2}\in \left( 0,1\right] .$ Then the
following inequality holds:%
\begin{equation}
\dfrac{1}{b-a}\int_{a}^{b}f\left( x\right) g\left( x\right) dx\leq \min
\left\{ E_{1},E_{2}\right\}  \label{E28}
\end{equation}%
where%
\begin{eqnarray*}
E_{1} &=&\dfrac{1}{2}\left[ \dfrac{1}{2\alpha _{1}+1}f^{2}\left( a\right) +%
\dfrac{2\alpha _{1}^{2}}{\left( \alpha _{1}+1\right) \left( 2\alpha
_{1}+1\right) }m_{1}^{2}f^{2}\left( \dfrac{b}{m_{1}}\right) \right. \\
&&+\dfrac{2\alpha _{1}}{\left( \alpha _{1}+1\right) \left( 2\alpha
_{1}+1\right) }m_{1}f\left( a\right) f\left( \dfrac{b}{m_{1}}\right) +\dfrac{%
1}{2\alpha _{2}+1}g^{2}\left( a\right) \\
&&+\dfrac{2\alpha _{2}^{2}}{\left( \alpha _{2}+1\right) \left( 2\alpha
_{2}+1\right) }m_{2}^{2}g^{2}\left( \dfrac{b}{m_{2}}\right) \\
&&\left. +\dfrac{2\alpha _{2}}{\left( \alpha _{2}+1\right) \left( 2\alpha
_{2}+1\right) }m_{2}g\left( a\right) g\left( \dfrac{b}{m_{2}}\right) \right]
,
\end{eqnarray*}%
\begin{eqnarray*}
E_{2} &=&\dfrac{1}{2}\left[ \dfrac{1}{2\alpha _{1}+1}f^{2}\left( b\right) +%
\dfrac{2\alpha _{1}^{2}}{\left( \alpha _{1}+1\right) \left( 2\alpha
_{1}+1\right) }m_{1}^{2}f^{2}\left( \dfrac{a}{m_{1}}\right) \right. \\
&&+\dfrac{2\alpha _{1}}{\left( \alpha _{1}+1\right) \left( 2\alpha
_{1}+1\right) }m_{1}f\left( b\right) f\left( \dfrac{a}{m_{1}}\right) +\dfrac{%
1}{2\alpha _{2}+1}g^{2}\left( b\right) \\
&&+\dfrac{2\alpha _{2}^{2}}{\left( \alpha _{2}+1\right) \left( 2\alpha
_{2}+1\right) }m_{2}^{2}g^{2}\left( \dfrac{a}{m_{2}}\right) \\
&&\left. +\dfrac{2\alpha _{2}}{\left( \alpha _{2}+1\right) \left( 2\alpha
_{2}+1\right) }m_{2}g\left( b\right) g\left( \dfrac{a}{m_{2}}\right) \right]
.
\end{eqnarray*}
\end{theorem}

\begin{proof}
Since $f$ is $\left( \alpha _{1},m_{1}\right) -$convex function and $g$ is $%
\left( \alpha _{2},m_{2}\right) -$convex function, then we have 
\begin{equation}
f\left( ta+\left( 1-t\right) b\right) \leq t^{\alpha _{1}}f\left( a\right)
+m_{1}\left( 1-t^{\alpha _{1}}\right) f\left( \frac{b}{m_{1}}\right)
\label{E29}
\end{equation}%
and%
\begin{equation}
g\left( ta+\left( 1-t\right) b\right) \leq t^{\alpha _{2}}g\left( a\right)
+m_{2}\left( 1-t^{\alpha _{2}}\right) g\left( \frac{b}{m_{2}}\right)
\label{E30}
\end{equation}%
for all $t\in \left[ 0,1\right] $. It is easy to observe that%
\begin{equation}
\int_{a}^{b}f\left( x\right) g\left( x\right) dx=\left( b-a\right)
\int_{0}^{1}f\left( ta+\left( 1-t\right) b\right) g\left( ta+\left(
1-t\right) b\right) dt.  \label{E31}
\end{equation}%
Using the elementary inequality $cd\leq \frac{1}{2}\left( c^{2}+d^{2}\right)
\;$($c,d\geq 0$ reals), (\ref{E29}) and (\ref{E30}) on the right side of (%
\ref{E31}) and making the charge of variable and since $f,g$ is
non-increasing, we have%
\begin{eqnarray}
&&\int_{a}^{b}f\left( x\right) g\left( x\right) dx  \label{E32} \\
&\leq &\dfrac{1}{2}\left( b-a\right) \int_{0}^{1}\left[ \left\{ f\left(
ta+\left( 1-t\right) b\right) \right\} ^{2}+\left\{ g\left( ta+\left(
1-t\right) b\right) \right\} ^{2}\right] dt  \notag \\
&\leq &\dfrac{1}{2}\left( b-a\right) \int_{0}^{1}\left[ \left( t^{\alpha
_{1}}f\left( a\right) +m_{1}\left( 1-t^{\alpha _{1}}\right) f\left( \dfrac{b%
}{m_{1}}\right) \right) ^{2}\right.  \notag \\
&&+\left. \left( t^{\alpha _{2}}g\left( a\right) +m_{2}\left( 1-t^{\alpha
_{2}}\right) g\left( \dfrac{b}{m_{2}}\right) \right) ^{2}\right] dt  \notag
\end{eqnarray}%
\begin{eqnarray*}
&=&\dfrac{1}{2}\left( b-a\right) \left[ \dfrac{1}{2\alpha _{1}+1}f^{2}\left(
a\right) \right. \\
&&+\dfrac{2\alpha _{1}^{2}}{\left( \alpha _{1}+1\right) \left( 2\alpha
_{1}+1\right) }m_{1}^{2}f^{2}\left( \dfrac{b}{m_{1}}\right) \\
&&+\dfrac{2\alpha _{1}}{\left( \alpha _{1}+1\right) \left( 2\alpha
_{1}+1\right) }m_{1}f\left( a\right) f\left( \dfrac{b}{m_{1}}\right) +\dfrac{%
1}{2\alpha _{2}+1}g^{2}\left( a\right) \\
&&+\dfrac{2\alpha _{2}^{2}}{\left( \alpha _{2}+1\right) \left( 2\alpha
_{2}+1\right) }m_{2}^{2}g^{2}\left( \dfrac{b}{m_{2}}\right) \\
&&\left. +\dfrac{2\alpha _{2}}{\left( \alpha _{2}+1\right) \left( 2\alpha
_{2}+1\right) }m_{2}g\left( a\right) g\left( \dfrac{b}{m_{2}}\right) \right]
.
\end{eqnarray*}%
Analogously we obtain%
\begin{eqnarray}
&&\int_{a}^{b}f\left( x\right) g\left( x\right) dx  \label{E33} \\
&\leq &\dfrac{1}{2}\left( b-a\right) \left[ \dfrac{1}{2\alpha _{1}+1}%
f^{2}\left( b\right) \right.  \notag \\
&&+\dfrac{2\alpha _{1}^{2}}{\left( \alpha _{1}+1\right) \left( 2\alpha
_{1}+1\right) }m_{1}^{2}f^{2}\left( \dfrac{a}{m_{1}}\right)  \notag \\
&&+\dfrac{2\alpha _{1}}{\left( \alpha _{1}+1\right) \left( 2\alpha
_{1}+1\right) }m_{1}f\left( b\right) f\left( \dfrac{a}{m_{1}}\right) +\dfrac{%
1}{2\alpha _{2}+1}g^{2}\left( b\right)  \notag \\
&&+\dfrac{2\alpha _{2}^{2}}{\left( \alpha _{2}+1\right) \left( 2\alpha
_{2}+1\right) }m_{2}^{2}g^{2}\left( \dfrac{a}{m_{2}}\right)  \notag \\
&&\left. +\dfrac{2\alpha _{2}}{\left( \alpha _{2}+1\right) \left( 2\alpha
_{2}+1\right) }m_{2}g\left( b\right) g\left( \dfrac{a}{m_{2}}\right) \right]
.  \notag
\end{eqnarray}%
Rewriting (\ref{E32}) and (\ref{E33}), we get the required inequality in (%
\ref{E28}). The proof is complete.
\end{proof}

\begin{remark}
\label{r3} In Theorem \ref{t5}, if we choose $\alpha _{1}=\alpha _{2}=1$, we
obtain the inequality of Theorem \ref{t4}.
\end{remark}

\end{document}